\documentclass{amsart}
\usepackage{graphicx}
\usepackage{amssymb}
\usepackage{epstopdf}
\usepackage{nicefrac}
\usepackage{enumerate}
\usepackage{hyperref}
\usepackage{float}
\usepackage{color}
\usepackage{pst-all}
\usepackage{tabularx}

\input xy
\xyoption {all}

\newtheorem{thm}{THEOREM}[section]

\newtheorem{cor}[thm]{COROLLARY}
\newtheorem{defn}[thm]{DEFINITION}

\newtheorem{lemma}[thm]{LEMMA}

\newtheorem{prop}[thm]{PROPOSITION}

\newtheorem{remark}[thm]{REMARK}

\newcommand\myeq{\mathrel{\stackrel{\makebox[0pt]{\mbox{\normalfont\tiny def}}}{=}}}

\newcommand{\ds}{\displaystyle}

\newcommand{\F}{{\mathcal F}}

\newcommand{\cA}{{\mathcal A}}

\newcommand{\cG}{{\mathcal G}}

\newcommand{\cP}{{\mathcal P}}
\newcommand{\cQ}{{\mathcal Q}}

\newcommand{\cS}{{\mathcal S}}

\newcommand{\fG}{\mathfrak{G}}
\newcommand{\fM}{{\mathfrak{M}}}

\newcommand{\fX}{{\mathfrak{X}}}

\newcommand{\Homeo}{{\rm Homeo}}

\newcommand{\mR}{{\mathbb R}}
\newcommand{\mS}{{\mathbb S}}
\newcommand{\mT}{{\mathbb T}}
\newcommand{\mZ}{{\mathbb Z}}

\newcommand{\vp}{{\varphi}}

\newcommand{\whH}{\widehat{H}}
\newcommand{\wtH}{\widetilde{H}}

\newcommand{\wtpi}{{\widetilde{\pi}}}
\newcommand{\wtY}{\widetilde{Y}}

\newcommand{\psg}{{\rm pseudo}{\star}{\rm group}}

\newcommand{\wtM}{\widetilde{M}}
\newcommand{\wtf}{\widetilde{f}}

\newcommand\mor{\mathrel{\stackrel{\makebox[0pt]{\mbox{\normalfont\tiny t}}}{\sim}}}

\begin{document}

\title{Classifying matchbox manifolds}


\author{Alex Clark}
\thanks{AC and OL supported in part by EPSRC grant EP/G006377/1}
\address{Alex Clark, Department of Mathematics, University of Leicester, University Road, Leicester LE1 7RH, United Kingdom}
\email{adc20@le.ac.uk}

\author{Steven Hurder}
\address{Steven Hurder, Department of Mathematics, University of Illinois at Chicago, 322 SEO (m/c 249), 851 S. Morgan Street, Chicago, IL 60607-7045}
\email{hurder@uic.edu}

\author{Olga Lukina}
\address{Olga Lukina, Department of Mathematics, University of Illinois at Chicago, 322 SEO (m/c 249), 851 S. Morgan Street, Chicago, IL 60607-7045}
\email{lukina@uic.edu}

 \thanks{Version date: October 21, 2017. This work is an extensive revision of the authors'  unpublished preprint \cite{CHL2013c}. }

\date{}


\begin{abstract}
Matchbox manifolds are foliated spaces with totally disconnected   transversals. Two matchbox manifolds which are homeomorphic have return equivalent dynamics, so that invariants of return  equivalence can be applied to distinguish non-homeomorphic matchbox manifolds.  In this work we study the problem of showing the converse implication: when does return equivalence imply homeomorphism?  For the class of weak solenoidal matchbox manifolds, we show that if the base manifolds satisfy a strong form  of the Borel Conjecture, then return equivalence for the dynamics of their foliations implies the total spaces are homeomorphic.    In particular, we show that two equicontinuous $\mT^n$--like matchbox manifolds of the same dimension   are homeomorphic if and only if their corresponding restricted pseudogroups are return equivalent.   At the same time, we show that these results cannot be extended to include the ``\emph{adic}-surfaces'', which are a class of weak solenoids fibering over a closed surface of genus 2.
\end{abstract}

\maketitle



\section{Introduction} \label{sec-intro}

 A matchbox manifold     is a compact, connected metrizable space $\fM$, equipped with a decomposition   into leaves of constant dimension, 
 so that the pair $(\fM, \F)$ is a foliated space  as defined in \cite{CandelConlon2000,MS2006},    for which the local transversals to the foliation  are totally disconnected. 
   In particular,    
the leaves of $\F$ are the path connected components of $\fM$.     A matchbox manifold with $2$-dimensional leaves is a lamination by surfaces in the sense of Ghys  \cite{Ghys1999} and Lyubich and Minsky \cite{LM1997}.  The  ``solenoidal spaces'' of  Sullivan in \cite{Sullivan2014,Verjovsky2014} are examples of matchbox manifolds.  The dynamical and topological properties of matchbox manifolds have been studied in a series of works by the authors \cite{ClarkHurder2013,CHL2014,CHL2017a}.

  Matchbox manifolds arise naturally    as   exceptional minimal sets for foliations of compact manifolds, for example see \cite{Hurder2008,Hurder2014};    as the tiling spaces associated to  repetitive, aperiodic tilings of Euclidean space $\mR^n$ which have      finite local complexity,   for example see \cite{AP1998,Sadun2003,Sadun2008};  and they appear naturally in the study of group representation theory and index theory for leafwise elliptic operators for foliations, as discussed in the books \cite{CandelConlon2000,MS2006}.  
The classification problem for matchbox manifolds asks for invariants which distinguish their homeomorphism types.       
      For example, in  the study of aperiodic tilings and their invariants,    the cohomology and 
  K-Theory groups of their associated   tiling spaces have been calculated in many instances, as for example   in \cite{AP1998,BS2007,BargeSadun2011,ClarkHunton2016,FHK2002}.

A matchbox manifold $(\fM, \F)$ is  also  a type of dynamical system, as discussed   in \cite{Hurder2014}  for example. 
A homeomorphism between matchbox manifolds preserves the leaves, as they   are the path connected components of $\fM$, and thus  many dynamical properties of $\F$ are invariants of the homeomorphism class of $\fM$. For example, the foliation $\F$ is said to be \emph{minimal} if each leaf $L \subset \fM$ is dense, and this property is clearly a homeomorphism invariant. For a   clopen  transversal $W$ of $\F$,  the  dynamical properties  of a minimal foliation $\F$  are determined by     the pseudogroup $\cG_W$ of local holonomy maps acting on the transversal  $W$. 
 \emph{Return equivalence} of pseudogroup actions on Cantor spaces   is the analog   of the   notion of \emph{Morita equivalence} for  groupoids associated to smooth foliations of compact manifolds, as discussed for example by Haefliger in \cite{Haefliger1984,Haefliger2002a}.  One then has the following result, whose proof follows along the same method as for the case of smooth foliations:

    \begin{thm} \label{thm-main0}
Let $\fM_1$ and $\fM_2$ be minimal matchbox manifolds. Suppose that there exists a homeomorphism $h \colon \fM_1 \to \fM_2$, then the holonomy pseudogroup actions associated to $\fM_1$ and $\fM_2$  are return equivalent.
\end{thm}

    Now consider $\fM_1$ and $\fM_2$ which are minimal matchbox manifolds whose holonomy pseudogroups are return equivalent.  That is, assume    there exists       clopen transversals $W_1$ to $\fM_1$ and  $W_2$ to $\fM_2$, and a homeomorphism $h \colon W_1 \to W_2$ which conjugates the restricted holonomy actions.  
   It is natural to ask for assumptions   on $\fM_1$ and $\fM_2$ which are sufficient  to guarantee that the transverse map  $h$ extends to a homeomorphism $H \colon \fM_1 \to \fM_2$.
  In the case of $1$-dimensional flows,  there is    the following result of   Aarts and Oversteegen \cite[Theorem 17]{AO1995}:
\begin{thm}\label{thm-one-dim}
Two orientable,  minimal, $1$--dimensional matchbox manifolds are homeomorphic if and only if they are return equivalent.
\end{thm}
  Since any non--orientable, minimal, matchbox manifold admits an orientable double cover, this implies that the local dynamics   determines the global topology in dimension one. For a matchbox manifold with leaves of dimension greater than one, the question whether there exists a converse to Theorem~\ref{thm-main0} is much more subtle.  Julien and Sadun studied  in \cite{JS2015} the homeomorphism classification for the tiling spaces associated to aperiodic tilings of the Euclidean space $\mR^n$, and the relation to return equivalence for the   associated pseudogroups. 
  
 In this work, we consider  the converse   to   Theorem~\ref{thm-main0} when $\fM$ is   homeomorphic to a weak solenoid.     
 A \emph{weak solenoid} is defined as the inverse limit of an infinite sequence of   proper finite covering maps of a closed compact manifold, called the base of the solenoid. The properties of weak solenoids are recalled in Section~\ref{sec-solenoids}. In particular,  a weak solenoid is homeomorphic to the suspension of a minimal equicontinuous action of a finitely generated group on a Cantor set,  called the global monodromy for the solenoid.   In     Section~\ref{sec-bundles}, the problem of showing that a pair of weak solenoids which are return equivalent are also homeomorphic,   is reduced to showing that they have presentations with homeomorphic base manifolds  and conjugate     global holonomy actions. 
     

There  is a special class of solenoidal spaces where the converse to Theorem~\ref{thm-main0} can be proved without further assumptions.  
We say that $\cS_\cP$ is a \emph{toroidal solenoid} if it is  defined by a presentation $\cP$ as in \eqref{eq-defweaksol}, where each of the manifolds $M_{\ell}$ is homeomorphic to the $n$-torus $\mT^n$. The toroidal solenoids arise as the minimal sets for smooth foliations, as shown in \cite{ClarkHurder2011a}. 
For $n \geq 2$, we have the following generalization of Theorem~\ref{thm-main0}.

\begin{thm}\label{thm-main2}
Suppose that  $\fM_1$ and $\fM_2$ are homeomorphic to toroidal solenoids of the same dimension $n$.
 Then $\fM_1$ and $\fM_2$ are   homeomorphic if and only if the holonomy pseudogroup actions associated to $\fM_1$ and $\fM_2$  are return equivalent.
\end{thm}

For the  toroidal solenoids  with base dimension  $n = 1$,   the homeomorphism type of $\cS_{\cP}$ is determined by the asymptotic class of a sequence of integers $\{m_{\ell} \mid \ell > 0\}$, the covering indices, as shown by  Bing \cite{Bing1960} and  McCord \cite[Section~2]{McCord1965}, and see also Block and Keesling \cite[Corollary~2.6]{BlockKeesling2004}.
Moreover, Aarts and Fokkink showed in \cite[Section~3]{AartsFokkink1991} that the the asymptotic class of the sequence of covering indices $\{m_{\ell} \mid \ell > 0\}$   is determined by the  return equivalence  class of the flow.  This result will be discussed further in Section~\ref{sec-examples} below.

For the  toroidal solenoids  with base dimension  $n \geq 2$,  
the results of Giordano, Putnam and Skau in \cite{GPS2017},  and   Cortez and Medynets in   \cite{CortezMedynets2016}, provide complete invariants of the   return equivalence class of minimal equicontinuous free $\mZ^n$ actions on Cantor sets. Their invariants  combined with the conclusion of Theorem~\ref{thm-main2} yields a  classification of toroidal solenoids up to homeomorphism.

     In Section~\ref{sec-examples} below, we introduce the \emph{adic}-surfaces, which   are     $2$-dimensional weak solenoids, and give examples   of return equivalent   \emph{adic}-surfaces which are   non-homeomorphic.   For non-toroidal weak solenoids of  dimension greater than one,  it is necessary to impose  geometric conditions which rule out the examples such as given in Section~\ref{sec-examples},   in   order to obtain a converse to Theorem~\ref{thm-main0}. 

The first condition we impose is that there exists a leaf for the foliation  which is simply connected.  
Secondly, we    impose topological restrictions on the base manifolds, in order that their homeomorphism type of their proper coverings are  determined by their fundamental groups.  

Recall that a finite $CW$-complex  $Y$ is \emph{aspherical} if it is connected  and its universal covering space  is  contractible.    Let   $\cA$ denote the collection    of $CW$-complexes which are aspherical.
 Also recall  that the \emph{Borel Conjecture} is that if  $Y_1$ and $Y_2$ are homotopy equivalent, aspherical closed manifolds, then a homotopy equivalence between $Y_1$ and $Y_2$ is homotopic to a  homeomorphism between $Y_1$ and $Y_2$.    The Borel Conjecture has been proven for many classes of aspherical manifolds:
\begin{itemize}
\item the torus $\mT^n$ for all $n \geq 1$,
\item all  infra-nilmanifolds,
\item   closed Riemannian manifolds $Y$ with negative sectional curvatures,
\item   closed Riemannian manifolds $Y$ of dimension $n \ne 3,4$ with non-positive sectional curvatures. 
\end{itemize}
A compact connected manifold $Y$ is an \emph{infra-nilmanifold} if its universal cover $\wtY$ is contractible, and the fundamental group of $M$ has a nilpotent subgroup with finite index. 

 The above  list is not exhaustive. The history and current status of the Borel Conjecture is discussed in the surveys of Davis \cite{Davis2012} and L\"{u}ck  \cite{Luck2012}. We introduce the notion of a \emph{strongly Borel} manifold.
\begin{defn} \label{def-borel}
A collection $\cA_B$  of closed manifolds is called \emph{Borel} if it satisfies the   conditions
\begin{itemize}
\item[1)] Each $Y \in \cA_B$ is aspherical,
\item[2)] Any closed manifold $X$ homotopy equivalent to some $Y \in\cA_B$ is homeomorphic to $Y$, and
\item[3)] If $Y \in \cA_B,$ then any finite covering space of $Y $ is also in $\cA_B.$
\end{itemize}
We say that a closed manifold $Y$ is \emph{strongly Borel} if  the collection $\cA_Y \equiv \langle Y \rangle$ of all finite covers of $Y$ forms a Borel collection.
\end{defn}
Each class of manifolds in the above list is strongly Borel. Here is our second main result:

 \begin{thm}\label{thm-main3}
 Let $\cS_{\cP}$ and $\cS_{\cQ}$ be weak solenoids, for which the base manifolds $M_0$ of the presentation $\cP$ and  $N_0$ of the presentation $\cQ$ are both   strongly Borel closed manifolds of the same dimension. 
Assume that the foliations on  $\cS_{\cP}$ and $\cS_{\cQ}$  each contain      a   leaf which is    simply connected.
 Then $\cS_{\cP}$ and $\cS_{\cQ}$  are homeomorphic if and only if the holonomy pseudogroup actions associated to $\cS_{\cP}$ and $\cS_{\cQ}$  are return equivalent.
\end{thm}
 
The requirement that there exists a simply connected leaf implies that the global holonomy maps associated to each of these foliations are injective maps. This conclusion yields a connection between return equivalence for the foliations of $\cS_{\cP}$ and $\cS_{\cQ}$ and the homotopy types of the  approximating manifolds in the presentations $\cP$ and $\cQ$. This requirement need not be imposed for the case of $Y = \mT^n$ in Theorem~\ref{thm-main2}, due to the     algebraic    properties of $\mZ^n$. We also note that the injectivity of the global holonomy maps    implies that the fundamental groups $\pi_1(M_0, x_0)$ and $\pi_1(N_0, y_0)$ are residually finite.

    A key aspect of the hypotheses in   Theorems~\ref{thm-main2} and \ref{thm-main3}, is that the domains of the return equivalence can be taken to have arbitrarily small diameter. Consequently,     invariants of return equivalence developed to distinguish actions should have an asymptotic nature, in that they are defined for arbitrarily small transversals. 
 
 A homeomorphism between matchbox manifolds induces a quasi-isometry between the leaves of the respective foliations, equipped with the induced metrics. It  is a classical result of  Plante \cite{Plante1975} that the quasi-isometry class of a leaf is determined by its intersection with any transversal, and thus provides a general invariant of asymptotic return equivalence.  
 For example, bounds on the growth rates of the leaves  are return equivalence  invariants. This observation was used  in the work \cite{DHL2016b}  to give   growth restrictions on the leaves which imply that the weak solenoid is a homogeneous continuum.

The \emph{asymptotic discriminant}  for an equicontinuous minimal Cantor action   was defined in \cite{HL2017a}, and is an invariant of the return equivalence class of the action, essentially by its definition. It thus provides an invariant of the homeomorphism class of the weak solenoid. Using this asymptotic invariant, the constructions of examples of wild solenoids   in \cite[Section~9]{HL2017a} were  shown to yield uncountable collections of non-homeomorphic weak solenoids, all with the same compact base manifold whose fundamental group is a higher rank lattice, and in particular is highly non-abelian.

\section{Standard forms for weak solenoids}\label{sec-solenoids}

Weak solenoids were first introduced by McCord \cite{McCord1965},   and   we recall  here the definitions and some of their properties  as developed by    Schori \cite{Schori1966}, Rogers and Tollefson \cite{RT1971b,RT1972} and Fokkink and Oversteegen \cite{FO2002}. We then recall the ``odometer representation'' of a weak solenoid as the suspension of a (non-abelian) group odometer (or subodometer) action.

  A  \emph{presentation} (for a weak solenoid)   is a collection 
\begin{equation}\label{eq-defweaksol}
\cP = \{ p_{\ell+1} \colon M_{\ell+1} \to M_{\ell} \mid \ell \geq 0\} \ ,
\end{equation}
    where each $M_{\ell}$ is a connected compact manifold   of dimension $n$, and each  \emph{bonding} map $p_{\ell +1}$  is a proper covering map of finite index.
 The weak solenoid   $\cS_{\cP}$   is the inverse limit associated to the presentation $\cP$:  
\begin{equation}\label{eq-presentationinvlim}
\cS_{\cP} \equiv \lim_{\longleftarrow} ~ \{ p_{\ell +1} \colon M_{\ell +1} \to M_{\ell}\} ~ \subset \prod_{\ell \geq 0} ~ M_{\ell} ~ .
\end{equation}
 By definition, for a sequence $\{x_{\ell} \in M_{\ell} \mid \ell \geq 0\}$, we have 
$$
x = (x_{\ell}) \equiv (x_0, x_1, \ldots ) \in \cS_{\cP}   ~ \Longleftrightarrow  ~ p_{\ell}(x_{\ell}) =  x_{\ell-1} ~ {\rm for ~ all} ~ \ell \geq 1 ~. 
$$
The set $\cS_{\cP}$ is given  the relative  (or Tychonoff) topology induced from the product topology. Then $\cS_{\cP}$ is   compact and connected.
  McCord showed in \cite{McCord1965} that  the space $\cS_{\cP}$ has a local product structure, and moreover we have:

\begin{prop}\label{prop-solenoidsMM}
Let  $\cP$ be a presentation with base space $M_0$ of dimension $n \geq 0$, and let  $\cS_{\cP}$ be the associated weak solenoid. Then   $\cS_{\cP}$ is  a matchbox manifold of dimension $n$, and the leaves of the foliation   $\F_{\cS}$ are the   path-connected components of $\cS_{\cP}$. 
\end{prop}

Associated to a presentation $\cP$  is a sequence of proper surjective maps 
\begin{equation}\label{eq-coverings}
q_{\ell} = p_{1} \circ \cdots \circ p_{\ell -1} \circ p_{\ell} \colon M_{\ell} \to M_0 ~ .
\end{equation}
For each $\ell > 1$, projection onto the $\ell$-th factor in the product $\ds \prod_{\ell \geq 0} ~ M_{\ell}$ in \eqref{eq-presentationinvlim} yields a 
  fibration map denoted by $\Pi_{\ell} \colon \cS_{\cP}  \to M_{\ell}$, for which 
 $\Pi_0 = q_{\ell} \circ \Pi_{\ell}  \colon \cS_{\cP} \to M_0$.

Fix a choice of a basepoint $x_0 \in M_0$ and let   $\fX_0 = \Pi_0^{-1}(x_0)$ be the fiber over $x_0$. Then  $\fX_0$ is a Cantor set by the assumption  that the fibers of each map $p_{\ell}$ have cardinality at least $2$.  

Choose a basepoint $x \in \fX_0$, and for   $\ell \geq 1$, define basepoints  $x_{\ell} = \Pi_{\ell}(x) \in M_{\ell}$. Then let 
\begin{equation}\label{eq-imahes}
G^x_{\ell} = {\rm image}\left\{  (q_{\ell} )_{\#} \colon \pi_1(M_{\ell}, x_{\ell}) \longrightarrow   G_0\right\}  
\end{equation}
  denote  the image of the induced map $(q_{\ell} )_{\#} $ on fundamental groups. Associated to the presentation $\cP$ and basepoint $x \in \fX_0$ we thus obtain a descending chain of subgroups of finite index
  \begin{equation}\label{eq-descendingchain}
\cG^x \equiv  \{ G^x_{\ell} \}_{\ell \geq 0} = \left\{G_0 = G^x_0 \supset G^x_{1} \supset G^x_{2} \supset \cdots \supset G^x_{\ell} \supset \cdots \right\} \ .
\end{equation}
Each quotient  $X^x_{\ell} = G_0/G^x_{\ell}$ is   a   finite set equipped with a left $G_0$-action, and the natural surjections $X^x_{\ell +1} \to X^x_{\ell}$   commute with the action of $G_0$.  Thus, the inverse limit 
\begin{equation}\label{eq-Galoisfiber}
X^x_{\infty} = \lim_{\longleftarrow} ~ \{ p_{\ell +1} \colon X^x_{\ell +1} \to X^x_{\ell}\} ~ \subset \prod_{\ell \geq 0} ~ X^x_{\ell}  
\end{equation}
is a $G_0$-space. Give $X^x_{\infty}$ the relative topology induced from   the product (Tychonoff) topology on the   space $\ds \prod_{\ell \geq 0} ~ X^x_{\ell}$, so that $\ds X^x_{\infty}$   is   a totally disconnected perfect compact set, so is   a Cantor space. 

Note that the subgroups $G^x_{\ell}$ in \eqref{eq-imahes} $\ds X^x_{\infty}$ are not assumed be normal in $G_0$, and thus $X^x_{\infty}$ is not a profinite group in general, without  some form of  ``normality'' assumptions on the subgroups in the chain   $\cG^x$. The question of what assumptions are necessary for the limit $\ds X^x_{\infty}$ to be a profinite group was first raised in the work \cite{RT1971b} by Rogers and Tollefson, and further analyzed by Fokkink and Oversteegen  in \cite{FO2002}. The  subsequent   work by Dyer, Hurder and Lukina in \cite{DHL2016a} characterized the necessary normality condition in terms of the discriminant invariant of the chain $G^x_{\ell}$.

A sequence $(g_{\ell})  \subset G_0$ such that $g_{\ell} G^x_{\ell} = g_{\ell+1} G^x_{\ell}$ for all $\ell \geq 0$ determines a point $(g_{\ell} G^x_{\ell}) \in X^x_{\infty}$. 
Let $e \in G_0$ denote the identity element, then the sequence $e_{0} = (e G^x_{\ell})$ is the \emph{standard basepoint} of $X^x_{\infty}$.
The action  $\Phi_x \colon G_0 \times X^x_{\infty} \to X^x_{\infty}$ is given by coordinate-wise multiplication, $\Phi_x(g)(g_{\ell} G^x_{\ell}) = (g g_{\ell} G^x_{\ell})$.

 We then have the standard observation:

 \begin{lemma}\label{lem-denseaction}
 $\Phi_x \colon G_0 \times X^x_{\infty} \to X^x_{\infty}$ defines an equicontinuous     Cantor minimal system  $(X^x_{\infty} , G_0 , \Phi_x)$.  
\end{lemma}  
 
 When $X^x_{\infty}$ has the structure of a profinite group, the    action  $\Phi_x \colon G_0 \times X^x_{\infty} \to X^x_{\infty}$ is called an \emph{odometer} by Cortez and Petite in \cite{CortezPetite2008}, and when $X^x_{\infty}$ is simply a Cantor space they call the action a    \emph{subodometer}. If the group $G_0$ is abelian, then $X^x_{\infty}$ is  a profinite abelian group, and more generally 
 if the chain   \eqref{eq-descendingchain}  consists of normal subgroups of $G_0$, then $X^x_{\infty}$ is a profinite   group. 
 For simplicity, we  will   call all of these equicontinuous minimal actions by the   nomenclature ``odometers''.

 Recall that $\Pi_{\ell} \colon \cS_{\cP}  \to M_{\ell}$ is a fibration for each $\ell \geq 0$, and so   the set $\fX^x_{\ell} = \Pi_{\ell}^{-1}(x_{\ell})$ is a clopen subset of $\fX_0$.
 From the relation   $q_{\ell+1}  \circ \Pi_{\ell+1} = \Pi_{\ell}$ we have that $\fX^x_{\ell +1} \subset \fX^x_{\ell}$ so we obtain a nested chain of clopen subsets 
 $\{\fX^x_{\ell +1} \subset \fX^x_{\ell} \mid \ell \geq 0\}$. Moreover, by the definition of the topology on the inverse limit $\cS_{\cP}$,   the intersection of these sets is the chosen basepoint $x \in \fX_0$.
 
  The global monodromy  action  $\Phi_{\F} \colon G_0 \times \fX_0 \to \fX_0$   is then defined as follows. Given a point $y \in \fX^x$, let $L_y \subset \cS_{\cP}$ be the leaf containing $y$.  
The restriction $\Pi_0 \colon L_y  \to M_0$ is a covering map, so  given a closed path $\sigma \colon  [0,1] \to M_0$ with basepoint $x_0$,   there is a unique leafwise path $\sigma_y$ in $L_y$ with initial point $y$ and terminal point   $\sigma_y(1) \in \cS_{\cP}$. The terminal point $\sigma_y(1)$ depends only on the basepoint-preserving homotopy class of the path $\sigma$.
 Given $g \in G_0$ and $y \in \fX_0$ choose a closed path $\sigma^g$ in $M_0$ representing $g$, choose a lift $\sigma^g_y$ as above, then set $\Phi_{\F}(g)(y) =  \sigma^g_y(1)$. This yields a well-defined group action of $G_0$ on the Cantor space $\fX_0$.

The subgroup   $G^x_{\ell} \subset G_0 = \pi_1(M_0 , x_0)$ is represented by closed paths in $M_0$ with basepoint $x_0$ and which admit a   lift for the covering  $q_{\ell} \colon M_{\ell} \to M_0$ to a closed path with endpoint $x_{\ell}$. It follows that for the leaf $L_x \subset \cS_{\cP}$ containing $x \in \fX_0$, we can also characterize $G^x_{\ell}$ as the subgroup represented by those closed paths which admit a lift  to $L_x$ which start at $x$, and terminate at a point in $L_x \cap \fX^x_{\ell}$. 
Thus  we have
  \begin{equation}\label{eq-cosets}
G^x_{\ell} = \{ g \in G_0 \mid \Phi_{\F}(g) ( \fX^x_{\ell} ) = \fX^x_{\ell}\} \ .
\end{equation}
 That is, the action $\Phi_{\F}$ of $g$ fixes the set $\fX_\ell^x$, possibly permuting points within this subset.

Let $g \in G_0$ represent the coset   $[g]_{\ell} \in G_0/G_\ell$. It follows from (7) that the image $\fX^{x,g}_\ell = \Phi_{\F}(g)(\fX_\ell^x)$ of $\fX^x_\ell$ under the action of $g$ either coincides with $\fX^x_\ell$ or it is disjoint from $\fX^x_\ell$. Thus the collection $\{\fX_\ell^{x,g}\}_{g \in G}$ is a finite collection of disjoint clopen sets which cover $\fX_0$.
Moreover, for all $\ell' > \ell > 0$, the collection of clopen sets 
 $\ds \{ \fX_{\ell'}^{x,g}   \mid [g]_{\ell'} = g G^x_{\ell'} \in   G^x_{\ell}/G^x_{\ell'} \}$ is a finite partition of $X^x_{\ell}$.
 
Given $y \in \fX_0$ there exists  a unique  $(g_\ell G_\ell) \in X^x_\infty$   so that 
 $\ds y = \bigcap_{\ell \geq 0} \ \fX_{\ell}^{x,g_{\ell}}$. Define $\sigma_x \colon \fX_0 \to X^x_{\infty}$ by $\sigma_x(y) = (g_\ell G_\ell)$.
 The map $\sigma_x$ is surjective, bijective and continuous, hence a homeomorphism. 
 Define 
 $\tau_x = \sigma_x^{-1} \colon X^x_{\infty} \to  \fX_0$ so that $\tau_x(e_0) = x$.  The map $\tau_x$ can be viewed as ``coordinates'' on $\fX_0$ centered at the chosen basepoint $x \in \fX_0$. 
  It follows from the construction of $\tau_x$ that it commutes with the left $G_0$-actions $\Phi_{\F}$   on $\fX_0$ and $\Phi_x$  on $X^x_{\infty}$.

The group chain \eqref{eq-descendingchain} and the homeomorphism $\tau_x$ depend on the choice of a point $x \in \fX_0$. For a different basepoint $y \in \fX_0$ in the fibre over $x_0$, for each $\ell > 0$ there exists $g_{\ell} \in G_0$ such that $y \in \fX^y_{\ell} \equiv \Phi_{\F}(g_{\ell}) ( \fX^x_{\ell} )$, 
and hence $\ds y = \bigcap_{\ell \geq 0} \ \fX_{\ell}^{y}$. 
Then for each $\ell > 0$, define $G_{\ell}^y = g_{\ell} G_{\ell}^x g_{\ell}^{-1}$ which consists of elements of $G_0$ that leave the set $\fX^y_{\ell}$ invariant. 
Let $\cG^y = \{G^y_{\ell} \mid \ell \geq 0 \}$ be the resulting group chain, with corresponding inverse limit space $X^y_{\infty}$.
Then the map   $\tau_y \colon X^y_{\infty} \to  \fX_0$  gives  coordinates  on $\fX_0$ centered at the chosen basepoint $y \in \fX_0$.

    The composition $\tau_y \circ \tau_x^{-1} \colon X_{\infty}^x \to X_{\infty}^y$ gives a topological conjugacy between the minimal Cantor actions  $(X_{\infty}^x , G_0 , \Phi_x)$ and $(X_{\infty}^y , G_0 , \Phi_y)$, and the composition  $\tau_y \circ \tau_x^{-1}$ can be viewed as a ``change of coordinates''. Properties of the minimal Cantor action 
$(X_{\infty}^x , G_0 , \Phi_x)$ which are independent of the choice of these coordinates are thus properties of the topological type of $\cS_{\cP}$. 
 
  The group chains $\cG^y$ and $\cG^x$  are said to be \emph{conjugate chains}. This notion forms an equivalence relation on group chains  which was introduced by Fokkink and Oversteegen \cite{FO2002}. The properties of this equivalence relation were studied in depth in \cite{DHL2016a,DHL2016c}.

The map  $\tau_x \colon X^x_{\infty}  \to  \fX_0$ is used to give   the ``odometer model'' for the solenoid $\cS_{\cP}$.
Let $\wtM_0$ denote the universal covering of the compact manifold $M_0$, and let $(X^x_{\infty} , G_0 , \Phi_x)$ be the minimal Cantor system associated to the presentation $\cP$ and the choice of a basepoint $x \in \fX_0$. Associated to the left action $\Phi_x$ of $G_0$ on $X^x_{\infty}$ is a suspension space 
\begin{equation}\label{eq-suspensionfols}
\fM_{\Phi} = \wtM_0 \times X^x_{\infty} / (z \cdot g^{-1}, y) \sim (z , \Phi_x(g)(y)) \quad {\rm for }~ z \in \wtM_0 , ~ g \in G_0, ~ y \in X^x_{\infty}
\end{equation}
which is a minimal matchbox manifold.  This construction is a generalization of a standard technique for constructing smooth foliations, as discussed in \cite{CN1985,CandelConlon2000} for example. 

Moreover, the suspension space $\fM_{\Phi}$ of a minimal equicontinuous action $\vp$ has an inverse limit  presentation,  where all of the bonding maps between the coverings $M_{\ell} \to M_0$ are derived from the universal covering map $\wtpi \colon \wtM_0 \to M_0$. The  following result  is given in  \cite{ClarkHurder2013}, and its proof is a consequence of the lifting property for maps between coverings:
\begin{thm}   \label{thm-weaksuspensions}
Let   $\cS_{\cP}$ be a weak solenoid with base space $M_0$. Then the suspension of the map $\tau_x$ yields a foliated homeomorphism $\tau_x^* \colon \fM_{\Phi} \to \cS_{\cP}$.
\end{thm}
  
  \begin{cor}\label{cor-weaksuspensions}
The homeomorphism type of a weak solenoid    $\cS_{\cP}$ is completely determined by the base manifold $M_0$ and the associated minimal Cantor system  
$(X_{\infty}^x , G_0 , \Phi_x)$.
\end{cor}

We conclude this discussion of some basic geometry  of weak solenoids, by recalling some properties of the holonomy groups of the foliations of weak solenoids.
First,   recall a basic result of Epstein, Millet and Tischler \cite{EMT1977}.

\begin{thm} \label{thm-emt}
Let $(\fX , G, \Phi)$ be a given action, and suppose that $\fX$ is a Baire space. 
Then the  union of all $x \in \fX$ such that the germinal holonomy group  ${\rm Germ}(\Phi , x)$ at $x$ is   trivial   forms a  $G_{\delta}$ subset of $\fX$. 
\end{thm}
The main result in \cite{EMT1977} is stated in terms of the germinal holonomy groups of leaves of a foliation, but an inspection of the proof  shows that it  applies directly to a general action $(\fX , G, \Phi)$.

 We conclude   by introducing  the following important notion:
 \begin{defn}\label{def-kernel}
The \emph{kernel} of the group chain  $\cG^x = \{G^x_{\ell}\}_{\ell \geq 0}$ is the subgroup  $\ds K(\cG^x) = \bigcap_{\ell \geq 0} ~ G^x_{\ell}$.
\end{defn}
  
For a weak solenoid $\cS_{\cP}$ with choice   of a basepoint $x_0 \in M_0$ and  fiber  $\fX_0 = \Pi_0^{-1}(x_0)$, 
  the kernel subgroup $K(\cG^x) \subset G_0$  may depend on the choice of the basepoint $x \in \fX_0$. 
  The dependence   of $K(\cG^x)$ on $x$ is a natural aspect  of the dynamics of the foliation $\F_{\cS}$ on $\cS_{\cP}$, when 
  $K(\cG^x)$ is interpreted in terms of the topology of the leaves of $\F_{\cS}$ as follows.
  
The map $\tau_x^* \colon \fM_{\Phi} \to \cS_{\cP}$ of Theorem~\ref{thm-weaksuspensions} sends the 
  quotient space  $\wtM/K(\cG^x)$    to the leaf $L_x \subset \cS_{\cP}$ through $x \in \fX_0$ in $\cS_{\cP}$, and so $K(\cG^x)$  is naturally identified with the fundamental group $\pi_1(L_x , x)$.
 The global holonomy homomorphism $\Phi_{\F,x} \colon \pi_1(L_x , x) \to Homeo(\fX_0 , x)$    of the leaf $L_x$ in the suspension foliation    $\F_{\cS}$ of $\cS_{\cP}$ is  then conjugate to the left action,  $\Phi_0 \colon K(\cG^x) \to Homeo(X^x_{\infty} , e_0)$.

 From the point of view of foliation theory, the leaves of $\F_{\cS}$ with holonomy are a ``small'' set by the proof of Theorem~\ref{thm-emt}.  There   always exists leaves without holonomy, while there may exist leaves with holonomy, and so the fundamental groups of the leaves may vary accordingly.  This aspect of the foliation dynamics of weak solenoids is discussed further in \cite[Section~4.2]{DHL2016c}.

 \section{Return equivalence}\label{sec-bundles}

The conclusion of Theorem~\ref{thm-weaksuspensions}  is that a weak solenoid is homeomorphic to a suspension space \eqref{eq-suspensionfols} of an equicontinuous action on a Cantor space. In this section, we consider  the notion of return equivalence between such suspension spaces.

Let   $\vp \colon G \times \fX \to \fX$ be a minimal    action on a Cantor space $\fX$. In order to give a precise definition of return equivalence, we 
introduce   the $\psg$ associated to the   action    $\vp$.  A more general  discussion of $\psg$s can be found in the works \cite{Hurder2014} and  \cite[Section~2.4]{HL2017a}.
 
 For each $g \in G$ and open subset $U \subset \fX$, let $\vp^U(g) \colon U \to V = \vp(g)(U)$ denote the restricted homeomorphism. Then the $\psg$ associated to $\vp$ is the collection of maps
 \begin{equation}
\Psi^*(\vp, \fX) \equiv \left\{ \vp^U(g) \mid U \subset \fX ~ {\rm open} ~ \ , ~ g \in G \right\} \ .
\end{equation}
The collection $\Psi^*(\vp, \fX)$ is not a pseudogroup, as it does not satisfy the   ``gluing'' condition on maps, but $\Psi^*(\vp, \fX)$ does generate   the usual pseudogroup $\Psi(\vp, \fX)$ associated to the action $\vp$   on $\fX$.
 
 Given an open subset $W \subset  \fX$, define the restriction of $\Psi^*(\vp, \fX)$  to $W$, 
 $$
 \Psi^*(\vp, W) =   \left\{ \vp^U(g) \mid U \subset W ~ {\rm open} ~  , ~ g \in G ~ , ~\vp(g)(U) \subset W \right\} \ .
 $$

\begin{defn}\label{def-repsg}
Let $\vp_i \colon G_i \times \fX_i \to \fX_i$ be minimal actions on Cantor spaces $\fX_i$ for $i=1,2$.
Then $\vp_1$ and $\vp_2$ are return equivalent   if there exists non-empty open sets $W_1 \subset \fX_1$ and $W_2 \subset \fX_2$, and a homeomorphism   $h \colon W_1 \to W_2$ which conjugates the restricted $\psg$ $\Psi^*(\vp_1, W_1)$ with the restricted $\psg$  $\Psi^*(\vp_2, W_2)$. 
\end{defn}

It is   an exercise to show that   minimal  suspension spaces $\fM_{\vp_1}$ and $\fM_{\vp_2}$ are return equivalent as foliated spaces, if and only if their associated global monodromy actions satisfy Definition~\ref{def-repsg}. 
 
We next introduce a notion which especially pertains to equicontinuous Cantor actions.

\begin{defn}\label{defn-adapted}
Let   $\vp \colon G \times \fX \to \fX$ be an action on a Cantor space $\fX$. 
A non-empty clopen subset $U \subset \fX$ is \emph{adapted} to the action $\vp$ if for any $g \in G$, 
$\vp(g)(U) \cap U \ne \emptyset$ implies that $\vp(g)(U) = U$. It follows that
 \begin{equation}
G_U = \left\{g \in G \mid \vp(g)(U) \cap U \ne \emptyset  \right\}  
\end{equation}
is a subgroup of $G$.
 \end{defn}

\begin{remark}\label{rmk-adapted}
{\rm 
For the action  $\Phi_x \colon G_0 \times X^x_{\infty} \to X^x_{\infty}$ of Lemma~\ref{lem-denseaction}, for each $\ell \geq 0$, the set $U = \fX^x_{\ell}$ is adapted with $G_U = G^x_{\ell}$ as defined in \eqref{eq-cosets}.
Note that if $V \subset U \subset \fX$ are both adapted to an action   $\vp \colon G \times \fX \to \fX$, with associated groups $G_V$ and $G_U$, then we have
$G_V = \{ g    \in  G_U \mid  \vp(g)(V) = (V)\}$. Moreover, if there exists a descending chain of clopen adapted sets $\{U_{\ell} \subset \fX \mid \ell \geq 0\}$ whose intersection is a point, then it is an exercise to show that the minimal action $\vp$ is equicontinuous. On the other hand, it is also easy to construct examples of actions which are not equicontinuous but admit a   proper adapted clopen subset $U \subset \fX$. For example,  consider any minimal Cantor action $\vp_U \colon G_U \times U \to U$, chose a non-trivial finite group $H$ and set  $G = H \times G_U$, then extend the action $\vp_U$ on $U$  to $\vp \colon G \times \fX \to \fX$  acting factor-wise  on the product space  $\fX = H \times U$.}
\end{remark}

We next establish two technical lemmas which are key for the proofs of Theorems~\ref{thm-main2} and \ref{thm-main3}.

\begin{lemma}\label{lem-key1}
Let   $\vp_i \colon G_i \times \fX_i \to \fX_i$ be    minimal actions on   Cantor spaces $\fX_i$, for $i=1,2$, and suppose there exists non-empty open sets 
$W_i \subset \fX_i$ and a homeomorphism   $h \colon W_1 \to W_2$ which conjugates the restricted $\psg$s $\Psi^*(\vp_{1}, W_{1})$ and $\Psi^*(\vp_{2}, W_{2})$. Then a clopen subset  $U_1 \subset W_1$ is adapted to the action $\vp_1$ if and only if $U_2 = h(U_1) \subset W_2$ is adapted to the action $\vp_2$.
\end{lemma}
\proof
We show that $U_2$ is adapted to the action of $\vp_2$. The reverse implication follows similarly. 

First note that $U_1$ is an open subset of $W_1$ and $h$ is a homeomorphism, hence $U_2$ is an open subset of $W_2$ in the relative topology on $\fX_2$ hence is an open subset of $\fX_2$. Also, $U_2$ is compact as $U_1$ is compact and all spaces are Hausdorff, thus $U_2$ is a clopen subset of $\fX_2$.

Let $g_2 \in G_2$ satisfy $\vp_2(g_2)(U_2) \cap U_2 \ne \emptyset$.
Let $h^* \colon \Psi^*(\vp_{2}, W_{2}) \to \Psi^*(\vp_{1}, W_{1})$ be the map induced by $h \colon W_1 \to W_2$ on the restricted $\psg$s. 
By assumption, this map is an isomorphism, and in particular  $h^*(\vp_2^{U_2}(g_2)) \in \Psi^*(\vp_{1}, W_{1})$.  Hence, there exists $g_1 \in G_1$ such that 
$\vp_1^{U_1}(g_1) = h^*(\vp_2^{U_2}(g_2))$. Thus, $\vp_1(g_1)(U_1) \cap U_1 \ne \emptyset$. As $U_1$ is adapted to the action of $\vp_1$ this implies that $\vp_1(g_1)(U_1) = U_1$, which implies that $\vp_2(g_2)(U_2) \cap U_2 = U_2$ as was to be shown.
\endproof

\begin{lemma}\label{lem-key2}
Let   $\vp \colon G \times \fX \to \fX$ be a  minimal action on a Cantor space $\fX$, and $U \subset \fX$ a clopen subset adapted to the action.
Then the collection  
$\cS_U \equiv \{\vp(g)(U) \mid g \in G\}$
forms a finite disjoint clopen partition of $\fX$.
\end{lemma}
\proof
We first show that the images form a disjoint partition. Suppose that for $g_1, g_2 \in G$ we have $\vp(g_1)(U) \cap \vp(g_2)(U) \ne \emptyset$.
Then $\vp(g_2^{-1} g_1)(U) \cap U \ne \emptyset$ hence $\vp(g_2^{-1} g_1)(U) = U$.
It follows that $\vp(g_1)(U) = \vp(g_2)(U)$. Each image $\vp(g_1)(U)$ is a clopen subset, and $\fX$ is compact, so there are only a finite number of disjoint images, which completes the proof.
\endproof

Assume that   $\vp \colon G \times \fX \to \fX$ is a  minimal action on a Cantor space $\fX$, and $U \subset \fX$ a clopen subset adapted to the action.
Let $p_U \colon \fX \to \cS_U$ be the natural map to the elements of the partition of $\fX$,  which exists by Lemma~\ref{lem-key2}.
Identify the collection $\cS_U$ with the quotient set $G/G_U$ via the   map $q_U(\vp(g)(U) ) = g G_U \in G/G_U$, then the composition 
$\pi_U = q_U \circ p_U \colon \fX \to G/G_U$ is $G$-equivariant.
  
Given an action $\vp \colon G \times \fX \to \fX$, we next construct the suspension foliated space for the action.  
Let   $M$ be a compact manifold without boundary, with     a basepoint $x_0 \in M$ and let $G = \pi_1(M, x_0)$ denote its fundamental group based at   $x_0$.  
 Let $\wtpi \colon \wtM \to M$ denote the universal covering space of $M$, defined by endpoint-fixed homotopy classes of paths in $M$ with initial point $x_0$.  Then $G$ acts on $\wtM$ on the right by deck transformations. 
  Define the quotient foliated space
\begin{equation}\label{eq-suspension}
\fM_{\vp} = (\wtM \times \fX) / \{(x \cdot \gamma, w) \sim (z , \vp(\gamma) \cdot w \} \quad , \quad z \in \wtM ~ , ~ w \in \fX ~, ~ \gamma \in G  \ .
\end{equation}
Let   $\pi \colon \fM_{\vp} \to M$ be the   map induced by the projection $\wtpi \colon \wtM \times \fX \to \wtM$ onto the first factor.

Now assume that the action $\vp$ admits a proper adapted clopen subset $U \subset \fX$. Then we   define 
\begin{equation}\label{eq-suspensionU}
M_{U} = (\wtM \times G/G_U) / \{(x \cdot g, w) \sim (z , g \cdot w \} \quad , \quad z \in \wtM ~ , ~ w \in G/G_U ~, ~ g \in G  \ .
\end{equation}
Note that $M_U$ is naturally identified with the finite covering space $\wtM/G_U$ of $M$ associated to the subgroup $G_U \subset G$.
Let $x_U \in M_U$ be the basepoint associated with the identity coset of $G/G_U$.

The quotient map $\pi_U  \colon \fX \to G/G_U$ induces a quotient map $\Pi_U \colon \fM_{\vp} \to M_U$ of suspension spaces, with     $U = \Pi_U^{-1}(x_U)  \subset \fX$, and there is  a commutative diagram:
 \begin{equation}\label{eq-collapse}
\xymatrix{
\fM_{\vp} \ar[d]_{\pi}\ar[rd]^{\Pi_U}& \\
M&M_{U}\ar[l]^{\pi_{G_U}}}
\end{equation}
 
 Note that the above construction applies to any minimal action with a proper adapted clopen subset. 
In the case where $U = \fX^x_{\ell}$ for an odometer action   $\vp = \Phi_x \colon G_0 \times X^x_{\infty} \to X^x_{\infty}$ and $\ell > 0$, then 
  $G_U = G^x_{\ell}$ as   in \eqref{eq-cosets} and the map fibration $\Pi_U$ is the same as the fibration $\Pi_{\ell}$   defined following \eqref{eq-coverings}.

We can now give a  result which is a key observation for the proofs of Theorems~\ref{thm-main2} and \ref{thm-main3}.
 For $i=1,2$, let   $\vp_i \colon G_i \times \fX_i \to \fX_i$ be a minimal action  on  the Cantor space $\fX_i$.
Let  $M_i$ be a compact manifold without boundary, with
   basepoint $x_i \in M_i$ and    $G_i = \pi_1(M_i, x_i)$   its fundamental group based at   $x_i$.  
Assume that the actions $\vp_1$ and $\vp_2$ return equivalent,   so there exists  open sets $W_i \subset \fX_i$ and a homeomorphism $h \colon W_1 \to W_2$ which conjugates the restricted $\psg$ $\Psi^*(\vp_1, W_1)$ with the restricted $\psg$  $\Psi^*(\vp_2, W_2)$.

Let $U_1 \subset W_1$ be a clopen subset which is adapted to the action $\vp_1$   then by Lemma~\ref{lem-key1} the image $U_2 = h(U_1)$ is a clopen subset adapted to the action $\vp_2$. 
For $i=1,2$, let 
$$G_{U_i}  = \left\{g \in G_i \mid \vp_i(g)(U_i) = U_i  \right\} \subset G_i$$
 be the stabilizer group of $U_i$ for the action $\vp_i$.  

The action  $\vp_i$ induces a homomorphism  $\vp_{U_i} \colon G_{U_i}  \to \Lambda_i \subset \Homeo(U_i)$ onto a subgroup $\Lambda_i$. 
Then the inverse of the restriction $h_{U_1} \colon U_1 \to U_2$ induces an isomorphism   $\lambda_h \colon \Lambda_1 \to \Lambda_2$. 

Let $\pi_{G_{U_i}} \colon  M_{U_i} \to M_i$ be the finite covering associated to $G_{U_i}$ with basepoint $x_{U_i} \in M_{U_i}$ over $x_i$.
A homeomorphism $f \colon M_{U_1} \to M_{U_2}$ is said to \emph{realize} $\lambda_h$ if the following diagram commutes:
\begin{align}\label{eq-realize}
\xymatrix{
\pi_1(M_{U_1} , x_{U_1})   \quad =   &  G_{U_1}  \ar[d]_{\vp_{U_1}}   \ar[r]^-{f_{\#}} & G_{U_2}   \ar[d]^{\vp_{U_2}}    &  = \quad  \pi_1(M_{U_2} , x_{U_2})    \\
   & \Lambda_1     \ar[r]^-{\lambda_h} & \Lambda_2 &   
}
\end{align}

By \eqref{eq-collapse} we can represent $\fM_i$ as a suspension space over $M_{U_i}$ with basepoint fiber $U_i$ 
and monodromy action $\vp_{U_i} \colon G_{U_i}  \to  \Homeo(U_i)$. Let $\wtf \colon \wtM_{U_1} \to \wtM_{U_2}$ denote the lift of $f$ to the universal covering spaces. Then the   product map 
\begin{equation}
\wtf \times h \colon \wtM_{_1} \times U_1 \to \wtM_{U_2} \times U_2
\end{equation}
is a homeomorphism, and intertwines   the diagonal actions of $G_1$ and $G_2$, so descends to a homeomorphism between $\fM_{\vp_1}$ and $\fM_{\vp_2}$.
We have thus shown:

\begin{prop}\label{prop-inducedhomeos} 
Suppose there exists a homeomorphism  $f \colon M_{U_1} \to M_{U_2}$ which realizes the isomorphism $\lambda_h \colon \Lambda_1 \to \Lambda_2$ between the  groups of fiber automorphisms induced by return equivalence. Then the suspension spaces 
$\fM_{\vp_1}$ and $\fM_{\vp_2}$ are homeomorphic.
\end{prop}

  \section{Proofs of main theorems}\label{sec-proofs}

In this section, we   use Proposition~\ref{prop-inducedhomeos} to obtain   proofs of  Theorems~\ref{thm-main2} and \ref{thm-main3}.
For $i =1,2$,  let $\fM_i$ be a matchbox manifold   homeomorphic to a weak solenoid $\cS_{\cP_i}$ defined by a presentation
 \begin{equation}
\cP_i   =    \{ p_{i, \ell+1} \colon M_{i, \ell+1} \to M_{i, \ell} \mid \ell \geq 0\} \ , 
\end{equation}
 where the base manifolds $M_{1,0}$ and $M_{2,0}$ both have dimension $n \geq 1$.    
 Let $\Pi_{\cP_i} \colon \cS_{\cP_i} \to M_{i,0}$   denote the projection onto the base manifold.  
 
 Let $x_{i,0} \in M_{i,0}$  be a basepoint, let $G_{i,0} = \pi_1(M_{i, 0} , x_{i,0})$ and set  $\fX_{\cP_i} = \Pi_{\cP_i}^{-1}(x_{i,0})$.  
  
   The assumption that the holonomy pseudogroups defined by the foliations on $\fM_1$ and $\fM_2$ are return equivalent implies that the foliations of $\cS_{\cP_1}$ and $\cS_{\cP_2}$ are return equivalent. This in turn implies that the global monodromy actions 
  $$ \Phi_{\cP_1} \colon  G_{1,0} \times \fX_{1} \to \fX_{1} ~ , ~  \Phi_{\cP_2} \colon  G_{2,0} \times \fX_{2} \to \fX_{2}$$ 
  are return equivalent in the sense of Definition~\ref{def-repsg}.
  That is, there exists open sets $W_1 \subset \fX_{1}$ and $W_2 \subset \fX_{2}$ and a homeomorphism $h \colon W_1 \to W_2$ which   conjugates the restricted $\psg$ $\Psi^*(\Phi_{\cP_1}, W_1)$ with the restricted $\psg$  $\Psi^*(\Phi_{\cP_2}, W_2)$.   

\subsection{Odometer models} \label{subsec-odometer}

Assume we are given weak solenoids  $\cS_{\cP_1}$ and $\cS_{\cP_2}$.
Then as shown in Theorem~\ref{thm-weaksuspensions}, we can assume that  the weak solenoids $\cS_{\cP_i}$ are homeomorphic to the suspension of odometer actions as in  \eqref{eq-suspensionfols}. To fix notation, recall the construction of the odometer actions.  
Choose a basepoint $x \in W_1 \subset \fX_{1}$, and set $y = h(x) \in W_2 \subset \fX_{2}$. Then     form the group chains corresponding to the presentations  $\cP_1$ at $x$,  and $\cP_2$ at $y$:
 \begin{eqnarray}
\cG^{x}_{\cP_1} \equiv  \{ G^{x}_{1,\ell} \}_{\ell \geq 0}   & = &     \left\{G_{1,0} \supset G^{x}_{1,1} \supset G^{x}_{1,2} \supset \cdots \supset G^{x}_{1,\ell} \supset \cdots \right\}   \label{eq-descendingchainP1}  \\
\cG^{y}_{\cP_2} \equiv  \{ G^{y}_{2,\ell} \}_{\ell \geq 0}   & = &     \left\{G_{2,0} \supset G^{y}_{2,1} \supset G^{y}_{2,2} \supset \cdots \supset G^{y}_{2,\ell} \supset \cdots \right\} \ . \label{eq-descendingchainP2}  
\end{eqnarray}

Let   $\Phi_1 \colon G_{1,0} \times X_{1,\infty} \to X_{1,\infty}$ be the odometer formed from the chain $\cG^{x}_{\cP_1}$ and let $\tau_{1,x} \colon X_{1,\infty} \to \fX_{\cP_1}$ be the $G_{1,0}$-equivariant homeomorphism constructed in  Section~\ref{sec-solenoids}. 
Then we have $\tau_{1,x}(e_{1,0}) = x$ where $e_{1,0} = (e G^{x}_{1,\ell})$ is the basepoint of $X_{1,\infty}$.
Moreover,  recall from  \eqref{eq-cosets} that for $\ell > 0$, we have
$$G^x_{1,\ell} = \{ g \in G_{1,0} \mid \Phi_{\cP_1}(g) ( \fX_{i,\ell} ) = \fX_{i,\ell}\} \ .$$

Similarly, let $\Phi_2 \colon G_{2,0} \times X_{2,\infty} \to X_{2,\infty}$ be the odometer formed from the chain $\cG^{y}_{\cP_2}$ and let $\tau_{2,y} \colon X_{2,\infty} \to \fX_{\cP_2}$ be the corresponding $G_{2,0}$-equivariant homeomorphism with $\tau_{2,y}(e_{2,0}) = y$.

   The preimage $\tau_{1,x}^{-1}(\fX^x_{1,\ell} )$ is identified with the clopen set
 \begin{equation}\label{eq-nbhdbasis}
U_{1,\ell} = \{ (g_{k} G^{x}_{1,k}) \mid k \geq 0 , g_0 = g_1 = \cdots = g_{\ell} \in G^{x}_{1,\ell} \} \subset X_{1,\infty} \ .
\end{equation}
The collection $\{\fX^x_{1,\ell} \mid \ell > 0\}$ is a neighborhood basis around the basepoint $x \in W_1$,   so there exists $\ell_1 > 0$ such that $U_{1,\ell} \subset \tau_{1,x}^{-1}(W_1)$ for $\ell \geq \ell_1$. Set $U_1 = U_{1,\ell_1}$ then the clopen subset $U_1$ is adapted to the action of $\Phi_1$ with stabilizer subgroup $G_{U_1} = G^x_{\ell_1}$  by Remark~\ref{rmk-adapted}. 
Thus, the action  $\Phi_1$ induces an epimorphism    $\Phi_{U_1} \colon G_{U_1}  \to \Lambda_1 \subset \Homeo(U_1)$.

The image $h \circ \tau_{1}(U_1) \subset  \fX_{2}$ is a clopen subset adapted to the action of $\Phi_{\cP_2}$ by Lemma~\ref{lem-key1}. 
Set $U_2 = \tau_2^{-1} \circ h \circ \tau_{1}(U_1) \subset X_{2,\infty}$, which is a clopen set adapted to the action $\Phi_2$. Let $G_{U_2} \subset G_{2,0}$ be the stabilizer group of $U_2$. Then the action  $\Phi_2$ induces an epimorphism    $\Phi_{U_2} \colon G_{U_2}  \to \Lambda_2 \subset \Homeo(U_2)$.
Moreover,  the homeomorphism  $\tau_2^{-1} \circ h \circ \tau_{1} \colon U_1 \to U_2$   induces an isomorphism     $\lambda_h \colon \Lambda_1 \to \Lambda_2$. 

\begin{remark}{\rm 
 Before continuing with the proofs of the main theorems, we recall an aspect of the equivalence of weak solenoids from \cite{FO2002} and which is discussed in detail in \cite{DHL2016a}. The basepoint $e_{2,0} \in V$ so there exists $\ell_2 > 0$ such that $V_{2,\ell} \subset V$ for $\ell \geq \ell_2$, where $V_{2,\ell}$ is defined as in \eqref{eq-nbhdbasis}.  For the action $\Phi_2$ the group $G_{2,\ell}$ stabilizes the clopen set $V_{2,\ell}$ and hence also stabilizes $V$.
 However, it need not be the case that $G_V$ is equal to one of the subgroups $G_{2,\ell}$. It is only possible to conclude that there exists some $\ell \geq \ell_2$ for which $G_{2,\ell} \subset G_V$. This corresponds to the fact that homeomorphic weak solenoids are defined by  group chains which are equivalent in the sense of \cite{DHL2016a,FO2002}, which is to say that their group chains are interlaced up to isomorphism.}
 \end{remark}
  
 By  Lemma~\ref{lem-key2}, the collection $\cS_2 \equiv \{\Phi_{2}(g)(U_2) \mid g \in G_{2,0}\}$ is a clopen partition of $X_{2,\infty}$.
We will  apply  Proposition~\ref{prop-inducedhomeos} to show that the suspension spaces $\fM_{\Phi_1}$ and $\fM_{\Phi_2}$ are homeomorphic. First, we must   construct  a map of fundamental groups $f_* \colon G_{U_1}  \to G_{U_2}$ so that the diagram \eqref{eq-realize} is satisfied, and then construct a homeomorphism $f \colon M_{U_1} \to M_{U_2}$ which induces the map $f_*$.     
 
\subsection{Proof of Theorem~\ref{thm-main2}} \label{subsec-thm2}
For $i=1,2$, we are given that $\cS_{\cP_i}$ is a  toroidal solenoid  whose base has  dimension $n$, so   $M_{i,0} =  \mT^n$ and hence $G_{i,0}  \cong \mZ^n$. The manifold $M_{U_i}$ is a covering of $M_{i,0}$ hence is also a torus, with fundamental group which we identify with $\mZ^n$. 
Introduce the subgroups
\begin{equation}
K_i = ker \{ \Phi_{U_i} \colon  G_{U_i} \to \Lambda_i \subset \Homeo(U_i) \}   \subset \mZ^n \ .
\end{equation}
Each $K_i$   is  a free abelian subgroup with rank $0 \leq r_i < n$, and there is    a commutative diagram:
\begin{align}\label{eq-abeliansquare}
\xymatrix{
K_1  \ar@{^{(}->}[r]  &  G_{U_1}  \ar@{->>}[r]^{\Phi_{U_1}} \ar@{.>}[d]_{f_*}   &  \Lambda_1  \ar[d]^{\lambda_h}_{\cong}  \\
 K_2 \ar@{^{(}->}[r]  & G_{U_2}  \ar@{->>}[r]^{\Phi_{U_2}} & \Lambda_2  
}
\end{align}

\begin{lemma}\label{lem-extension}
There exists a map $f_* \colon G_{U_1} \to G_{U_2}$ such that the diagram \eqref{eq-abeliansquare} commutes.
\end{lemma}
\proof
This follows because    $G_{U_1} \cong G_{U_2} \cong \mZ^n$ are free abelian groups, hence   projective $\mZ$-modules. 
We give the details of the construction of the map $f_*$.
Let $\{a_1, \ldots , a_d\} \subset \Lambda_1$ be a minimal set of generators for $\Lambda_1$, then 
$\{\lambda_h(a_1), \ldots , \lambda_h(a_d)\} \subset \Lambda_2$ is a minimal set of generators for $\Lambda_2$. 

Choose   $\{g_1, \ldots , g_{d}\} \subset G_{U_1}$ so that 
$a_i = \Phi_{U_1}(g_i)$ for $1 \leq i \leq d$. The kernel $K_1$ is free abelian, so we can extend this set to a basis $\{g_1, \ldots , g_{n}\}$ 
for $G_{U_1}$  where $\Phi_{U_1}(g_i)$   is the identity for $d <  i \leq n$.

Choose elements  $\{g'_1, \ldots , g'_{d}\} \subset G_{U_2}$ so that 
$\lambda_h(a_i) = \Phi_{U_2}(g'_i)$ for $1 \leq i \leq d$.
 Note that both $K_1$ and $K_2$ are   free abelian of rank $n-d$, so we can extend this set to a basis $\{g'_1, \ldots , g'_{n}\}$ 
for $G_{U_2}$  where $\Phi_{U_2}(g'_i)$   is the identity for $d <  i \leq n$.
 
 Define the group isomorphism  $\ds f_* \colon G_{U_1} \to G_{U_2}$ by specifying $f_*(g_i) = g'_i$ for $1 \leq i \leq n$. Then the diagram \eqref{eq-abeliansquare} commutes by our choices of these bases.
\endproof

 Finally, to complete the proof of Theorem~\ref{thm-main2}, observe that $f_*$ extends to a linear map  $\widehat{f_*} \colon \mR^n \to \mR^n$, and so induces a diffeomorphism of the quotient spaces   $f \colon \mT^n \to \mT^n$. Then  the hypotheses of Proposition~\ref{prop-inducedhomeos} are satisfied.

\subsection{Proof of Theorem~\ref{thm-main3}}\label{subsec-thm3}
The proof of Theorem~\ref{thm-main3} uses the geometric hypotheses on the foliations of the weak solenoids $\cS_{\cP_i}$ to show the existence of the map $f_*$ such that the diagram \eqref{eq-abeliansquare} commutes,  in place of the group extension arguments in the proof of Lemma~\ref{lem-extension}. In particular, we  assume that the foliations on  $\cS_{\cP_1}$ and $\cS_{\cP_2}$  each contain      a dense  leaf which is    simply connected. 
By the results of Section~\ref{subsec-odometer}, we can assume that $\cS_{\cP_1}$ and $\cS_{\cP_2}$ are represented as suspensions of odometer actions, and thus it suffices to show that   the hypotheses of Proposition~\ref{prop-inducedhomeos} are satisfied.

We assume that the odometer actions $\Phi_i \colon G_{i,0} \times X_{i,\infty} \to X_{\infty}$ are return equivalent, for $i=1,2$, and that open subsets $W_i \subset X_{i,0}$ are chosen so that the restricted $\psg$ $\Psi^*(\Phi_{1}, W_1)$ is conjugate to   the restricted $\psg$  $\Psi^*(\Phi_{2}, W_2)$.
Then let $U_i \subset W_i$ be chosen as above, with a homeomorphism $h \colon U_1 \to U_2$ conjugating the restricted actions
  $\Phi_{U_i} \colon G_{U_i}  \to \Lambda_i \subset \Homeo(U_i)$.

 Let $K_i \subset G_{U_i}$ denote the kernel of the map $\Phi_{U_i}$, and for  $z \in U_i$ define:
  \begin{equation}\label{eq-isotropyz}
K_i(z) = \{ g \in G_{U_i} \mid \Phi_{U_i}(g)(z) = z\} \ .
\end{equation}
Observe that $K_i \subset K_i(z)$ for all $z \in U_i$.

By the definition \eqref{eq-suspension} of the suspension space $\fM_{\Phi_{U_i}}$ the leaf $\ds L_z \subset \fM_{\Phi_{U_i}}$ defined by the point $z$  is homeomorphic to  the covering $\wtM_i/K_i(z) \to M_i$. By assumption, for each $i=1,2$ there exists $z \in U_i$ so that $L_z$ is simply connected, which implies that $K_i(z)$ is the trivial group, which implies that the kernel $K_i$ is also the trivial group. Thus,  the map $\Phi_{U_i} \colon G_{U_i}  \to \Lambda_i$ is an isomorphism. Define the map 
\begin{equation}
f_* \equiv  \Phi_{U_2}^{-1} \circ \lambda_h \circ \Phi_{U_1} \colon G_{U_1} \to G_{U_2} 
\end{equation}
 which is an isomorphism such that    the  diagram \eqref{eq-abeliansquare} commutes.

 By the hypotheses of Theorem~\ref{thm-main3} the manifolds $M_1$ and $M_2$ are both strongly Borel, hence their finite coverings $M_{U_1}$ and $M_{U_2}$ satisfy the Borel Conjecture. The map $f_*$ induces a homotopy equivalence between them, as both have contractible universal covering spaces. Then by the solution of the Borel Conjecture for these spaces, there exists a homeomorphism $f \colon M_{U_1} \to M_{U_2}$ which induces the map $f_*$ on their fundamental groups.
 This completes the proof of Theorem~\ref{thm-main3}.
 
 \medskip

 \begin{remark}
{\rm
Note that the choice of the  clopen set $U_i$   in the  above proofs can be chosen to have arbitrarily small diameter, and hence the degree of the corresponding covering map $\pi_{U_i} \colon M_{U_i} \to M_i$ in \eqref{eq-collapse} can be chosen to be arbitrarily large. 
  As remarked   in \cite{Davis2012}, the homeomorphism $f$ that is obtained from the solutions of the Borel Conjecture  can be assume to be smooth for a sufficiently large finite covering. It follows that the homeomorphism $h \colon \cS_{\cP_1} \to \cS_{\cP_2}$ obtained from Proposition~\ref{prop-inducedhomeos} can be chosen to be   smooth along leaves.
}
 \end{remark}

\section{Examples and counter-examples}\label{sec-examples}

In this section, we give several examples to illustrate the necessity of the hypotheses of Theorem~\ref{thm-main3}. 
We first recall a classical result, the classification of Vietoris  solenoids of dimension one. We then consider extensions of this construction to solenoids with dimension $n \geq2$ and give examples of solenoids which are return equivalent but not homeomorphic. These examples are   essentially the simplest possible constructions. Many other variants on their construction are clearly possible, especially for solenoids of dimensions greater than two, as briefly discussed in Section~\ref{subsec-higher}. 

\subsection{Vietoris solenoids}

A \emph{Vietoris solenoid}  \cite{vanDantzig1930,Vietoris1927}   is   a $1$-dimensional solenoid $\cS_{\cP}$, where each $M_{\ell}$ is a circle, and each $p_{\ell} \colon \mS^1 \to \mS^1$ in the presentation $\cP$ is an orientation preserving covering map of degree   $m_{\ell} \geq 2$.   
Let $\vec{m} = \{m_1, m_2, \ldots\}$ be the list of covering degrees for $\cP$. Then $\cS_{\cP}$ is also called an $\vec{m}$-\emph{adic} solenoid of dimension one, and denoted by $\cS(\vec{m})$.

Let  $\vec{m} = \{m_{\ell} \mid \ell \geq 1\}$   denote a sequence of positive integers with each $m_i \geq  2$. Set $m_0=1$, then 
define the  profinite group
\begin{eqnarray}
\fG_{\vec{m}} &\myeq &\underleftarrow{\lim}~ \left\{ \, q_{\ell+1} \colon \mZ/m_1\cdots m_{\ell+1}\mZ \to \mZ/m_0m_1\cdots m_{\ell}\mZ \mid \ell \geq 1 \, \right\} \label{eq-Madicgroup} \\
&=&\underleftarrow{\lim}~ \left\{ \mZ/\mZ \xleftarrow{m_1} \mZ/m_1\mZ \xleftarrow{m_2} \mZ/m_1m_2\mZ \xleftarrow{m_3} \mZ/m_1m_2m_3\mZ\xleftarrow{m_4}\cdots \right\} \nonumber
\end{eqnarray}
where $q_{\ell+1}$ is the quotient map of degree $m_{\ell+1}$. 
 Each of the profinite groups $\fG_{\vec{m}}$ contains a  copy of $\mZ$ embedded as a dense subgroup by $z \to ([z]_0,[z]_1,...,[z]_k,...),$ where $[z]_k$ corresponds to the class of $z$ in the  quotient group $\ds \mZ/m_0\cdots m_k\mZ$. There is a homeomorphism $a_{\vec{m}} \colon \fG_{\vec{m}} \to \fG_{\vec{m}}$  given by ``addition of $1$'' in each finite factor group.  The resulting action of $\mZ$ is denoted by 
 $\Phi_{\vec{m}} \colon \mZ \times  \fG_{\vec{m}} \to \fG_{\vec{m}}$. The dynamics of $a_{\vec{m}}$ acting on $\fG_{\vec{m}}$  is   referred to as an \emph{adding machine}, or equivalently as a (classical) \emph{odometer}. We then have the standard result:
 
 \begin{prop}
The Vietoris solenoid $\cS(\vec{m}) $ is   homeomorphic to the suspension $\fM_{\Phi_{\vec{m}}}$ of the odometer action $\Phi_{\vec{m}}$  with base manifold $M_0 = \mS^1$.
\end{prop}

 Two Vietoris solenoids $\cS_{\cP}$ and $\cS_{\cQ}$ are homeomorphic if and only if their presentations $\cP$ and $\cQ$ yield group chains as in \eqref{eq-descendingchain} which are equivalent. As all of these are chains of subgroups of the fundamental group $\mZ$ of $\mS^1$, the equivalence problem for these chains reduces to giving conditions on the sequences of integer covering degrees in $\cP$ and $\cQ$ which imply equivalence of the chains. There are two invariants of sequences which arise in the classification problem. First, consider the function which counts the total number of occurrences of a given prime in the sequence of integers $\vec{m}$.
\begin{defn}\label{def-primefunction}
Given a sequence of positive integers $\vec{m}$ as above, let $C_{\vec{m}}$ denote the function from the set of
prime numbers to the set of extended natural numbers $\{0, 1, 2, . . .,\infty\}$ given by
$$C_{\vec{m}}(p)=\sum_1^\infty m_i(p),$$
where $m_i(p)$ is the power of the prime $p$ in the prime factorization of $m_i.$
\end{defn}
That is, $C_{\vec{m}}(p) = k$ means that the prime $p$ occurs a total of $k$ times in the prime factorization of the integers in the sequence $\vec{m}$.

\begin{thm}
The Vietoris solenoids $\cS(\vec{m}) $ and $\cS(\vec{n})$  are homeomorphic \emph{as bundles over the base manifold $\mS^1$} if and only if $C_{\vec{m}}(p) = C_{\vec{n}}(p)$ for all primes $p$.
\end{thm}

Next, we recall the notion of ``tail equivalence'' on sequences. This notion was introduced by Bing in \cite{Bing1960}, and plays a basic role in the study of return equivalence for Vietoris solenoids  in  \cite{AartsFokkink1991}.   

\begin{defn}\label{def-tailequiv}
 Two infinite sets of integers, $\vec{m} = \{m_{\ell} \mid \ell \geq 1\}$ and $\vec{n} = \{n_{\ell} \mid \ell \geq 1\}$, are said to be \emph{tail equivalent}, and we write $\vec{m} \mor \vec{n}$, if there exists cofinite  subsequences $\vec{m}_* \subset \vec{m}$ and $\vec{n}_* \subset \vec{n}$ which  are in bijective correspondence. 
\end{defn}

The following observation is a direct consequence  of   Definitions~\ref{def-primefunction} and \ref{def-tailequiv}.
\begin{lemma}\label{lem-MorEquvSequ}
Two sequences of integers $\vec{m}$ and $\vec{n}$ as above are  tail equivalent   if and only if the following two conditions hold:
\begin{enumerate}
  \item For all but finitely many primes $p$, $C_{\vec{m}}(p)=C_{\vec{n}}(p)$, and
  \item for all primes $p$, $C_{\vec{m}}(p)= \infty$  if and only if $C_{\vec{n}}(p)=\infty$. 
\end{enumerate}
\end{lemma}

 The classification of  Vietoris solenoids up to homeomorphism by Bing     \cite{Bing1960} and McCord \cite{McCord1965}, and the study of return equivalence   by  Aarts and Fokkink   in \cite{AartsFokkink1991}  yields:   
   
\begin{thm}\cite{McCord1965,AartsFokkink1991}\label{thm-onedimSol}
The Vietoris solenoids $\cS(\vec{m}) $ and $\cS(\vec{n}) $ are homeomorphic  if and only if they are return equivalent, if and only if    $\vec{m}$ and $\vec{n}$ are tail equivalent. 
\end{thm}

\subsection{$\vec{m}$-adic solenoids of   dimension two} 

Let $\Sigma_g$ be a closed surface of genus $g \geq 1$, which is obtained by attaching $g$ torus handles $\mT^2 = \mS^1 \times \mS^1$ to the $2$-sphere $\mS^2$.  For example, $\Sigma_1$ is homeomorphic to  the $2$-torus $\mT^2$. Pick a basepoint $x_0 \in \Sigma_g$ and let $G_0 = \pi_1(\Sigma_g , x_0)$ be the fundamental group. Choose an epimorphism $a \colon G_0 \to \mZ$, which corresponds to a non-trivial class $[a] \in H^1(\Sigma_g ; \mZ)$ in integral homology.
 
 Let $\vec{m} = \{m_{\ell} \mid \ell \geq 1\}$ denote a sequence of integers with each $m_i \geq  2$,  and form the profinite $\vec{m}$-adic group $\fG_{\vec{m}}$ as in \eqref{eq-Madicgroup}. Let $\Phi_{\vec{m}}$ denote the odometer action of $\mZ$ described above. Extend this to an action of $G_0$ 
 \begin{equation}
\Phi_{\vec{m}}^a \colon G_0 \times  \fG_{\vec{m}} \to \fG_{\vec{m}} \quad: \quad \Phi_{\vec{m}}^a(g)(x) = \Phi_{\vec{m}}(a(g))(x) ~ , ~ g \in G_0 ~, ~ x \in  \fG_{\vec{m}}
\end{equation}

 \begin{defn} \label{def-madicsurface}
 The $\vec{m}$-adic surface  $\fM(\Sigma_{g}, a, \vec{m})$  is the suspension space \eqref{eq-suspension} associated to the the action $\Phi_{\vec{m}}^a$ with base $\Sigma_g$. 
 \end{defn}
 
 We note a consequence of the construction of   $\fM(\Sigma_{g}, a, \vec{m})$, which follows immediately from the fact that   the action $\Phi_{\vec{m}}^a$ is induced from the action $\Phi_{\vec{m}}$ and the results of \cite{AartsFokkink1991}:
 \begin{prop}\label{prop-surfadicMor}
Given closed orientable surfaces $\Sigma_{g_1}$ and $\Sigma_{g_2}$ of genus $g_i \geq 1$ for $i=1,2$,  epimorphisms $a_i \colon G_{i,0} \to \mZ$ and  sequences $\vec{m}$ and $\vec{n}$, then 
$\fM(\Sigma_{g_1} , a_1 ,  \vec{m})$  is return equivalent to $\fM(\Sigma_{g_2} , a_2 , \vec{n})$  if and only if $\vec{m}$ and $\vec{n}$ are tail equivalent.
\end{prop}

Finally, we consider the problem, given \emph{adic}-surfaces $\fM(\Sigma_{g_1} , a_1 , \vec{m})$ and $\fM(\Sigma_{g_2} , a_2, \vec{n})$ such that  $\vec{m}$ is tail equivalent to $\vec{n}$, when are they homeomorphic as matchbox manifolds? 
First, consider the case of genus $g_1 = g_2=1$ so that  $\ds \Sigma_{g_1}= \Sigma_{g_2} = \mT^2$. 
Then   Theorem~\ref{thm-main2} and Proposition~\ref{prop-surfadicMor} yield:
\begin{thm}\label{thm-torusclass}
The \emph{adic}-surfaces   $\fM(\mT^2, a_1 , \vec{m})$ and $\fM(\mT^2, a_2 , \vec{n})$ are homeomorphic if and only if $\vec{m}$ and $\vec{n}$ are tail equivalent.
\end{thm}

For the general case of \emph{adic}-surfaces where at least one base manifold has higher genus, we next give examples of weak solenoids which are return equivalent but not homeomorphic. Note that in these examples, their base manifolds are compact surfaces hence are strongly Borel, but all their leaves   have non-trivial fundamental groups, so the hypotheses of Theorem~\ref{thm-main3} are not satisfied.
\begin{thm}\label{thm-surfaces}
Let $\fM_1 = \fM(\Sigma_{g_1} , a_1 , \vec{m})$ and $\fM_2 = \fM(\Sigma_{g_2} , a_2 , \vec{n})$ be \emph{adic}-surfaces.
\begin{enumerate}
\item If $g_1 > 1$ and $g_2 = 1$, then $\fM_1$ and $\fM_2$ are never homeomorphic.
\item If $g_1 = g_2 > 1$  and $a_1 = a_2$, then $\fM_1$ and $\fM_2$ are   homeomorphic if and only if   $C_{\vec{m}}=C_{\vec{n}}$.
\item If $g_1 = g_2 > 1$ and $a_1 = a_2$, then there exists $\vec{m} \mor \vec{n}$   but  $\fM_1 \not\approx \fM_2$. 
\end{enumerate}
\end{thm}
\proof
First, recall that the Euler characteristic of the closed surface $\Sigma_g$ of genus $g \geq 1$  has Euler characteristic $\chi(\Sigma_g) = 2-2g$, and the Euler characteristic is multiplicative for coverings. That is, if $\Sigma_g'$ is a $k$-fold covering of $\Sigma_g$ then $\chi(\Sigma_g') = k \cdot \chi(\Sigma_g)$. 
In particular, for $g > 1$,    a proper covering $\Sigma_g'$ of $\Sigma_g$   is never homeomorphic to $\Sigma_g$. 

Next, each of the spaces $\fM_1$ and $\fM_2$ is homeomorphic to an inverse limit as in \eqref{eq-presentationinvlim}, 
\begin{eqnarray}
\fM_1 = \fM(\Sigma_{g_1} , a_1 , \vec{m}) & \cong &  \lim_{\longleftarrow} ~ \{ f_{\ell +1} \colon M_{\ell +1} \to M_{\ell}\} \  \label{eq-presentationinvlim1} \\
\fM_2 = \fM(\Sigma_{g_2} , a_2 , \vec{n}) & \cong &  \lim_{\longleftarrow} ~ \{ g_{\ell +1} \colon N_{\ell +1} \to N_{\ell}\} \  , \label{eq-presentationinvlim2}  
\end{eqnarray}
where $M_0 = \Sigma_{g_1}$ and $N_0 = \Sigma_{g_2}$. For $\ell > 0$, let $m_{\ell}$ denote the degree of the covering map $ f_{\ell}$ and   let $n_{\ell}$  denote the degree of the covering map $g_{\ell}$. 

 Now assume  there is a homeomorphism $H \colon \fM_1 \to \fM_2$.
 By the results  of Rogers and Tollefson in \cite{RT1971,RT1972}, the map $H$ is   homotopic to a homeomorphism $\whH$ which is induced by a map between the inverse limit representations  of  $\fM_1$ in  \eqref{eq-presentationinvlim1}  and   of  $\fM_2$ in \eqref{eq-presentationinvlim2}. Such a map has the following form: 
  
 There exists  an increasing integer-valued function $k \to \ell_k$ for $k \geq 0$, and continuous onto maps $\wtH_{k} \colon  M_{\ell_k}    \to N_k$ where   
the collection of maps $\{\wtH_{k} \mid k \geq k_0\}$ form a commutative diagram:
\begin{equation}\label{eq-commutingdiagram}
\xymatrix{
M_{\ell_0} \ar[d]_{\wtH_0}
& M_{\ell_1} \ar[l]_{f^{\ell_1}_{\ell_0}}\ar[d]^{\wtH_1}
& \cdots\ar[l]
& M_{\ell_k}\ar[d]^{\wtH_k}\ar[l]
& M_{\ell_{k+1}} \ar[l]_{f^{\ell_{k+1}}_{\ell_{k}}}\ar[d]^{\wtH_{k+1}}
& \cdots\ar[l]\\
N_0
& N_1\ar[l]^{g_1}
&\cdots\ar[l]
& N_k\ar[l]
& N_{k+1} \ar[l]^{g_k}
& \cdots\ar[l]}
\end{equation}
where the $f_k$ and $g_k$ are the bonding   maps in the inverse limit representations \eqref{eq-presentationinvlim1} and \eqref{eq-presentationinvlim2},  and $f^{\ell_{k+1}}_{\ell_k} = f_{\ell_k +1} \circ \cdots \circ f_{\ell_{k+1}}$ denotes the   composition of  bonding maps.  

All of the horizontal maps in the diagram  \eqref{eq-commutingdiagram} are covering maps by construction. Moreover, as the spaces $M_k$ and $N_k$ are closed surfaces, we can assume that all of the vertical maps in      \eqref{eq-commutingdiagram} are also covering maps.
Thus, the Euler classes of all surfaces there are related by the covering degrees of the maps. 
For example, $\chi(M_{\ell_k}) = d_k \cdot \chi(N_k)$  where $d_k$ is the covering degree of $\wtH_k$. 

To show 1) we assume that a homeomorphism $H$ exists, and so we have diagram  \eqref{eq-commutingdiagram}  as above. Observe that $g_2 = 1$ implies that $\chi(\Sigma_2) = \chi(\mT^2) = 0$, hence  $\chi(N_k) = 0$ for all  $k \geq 0$. Then as $d_k \geq 1$ for all $k$, 
we obtain  $\chi(M_{\ell_k}) = 0$. But this contradicts the assumption that $g_1 > 1$ hence $\chi(M_{\ell_k}) < 0$ as $M_{\ell_k}$ is a covering of $\Sigma_1$ which has $\chi(\Sigma_1) < 0$. Thus $\fM_1 \not\approx \fM_2$. 

 To show  2) first assume that   $C_{\vec{m}}(p)=C_{\vec{n}}(p)$ for all primes $p$. Then the odometer actions
  $\Phi_{\vec{m}} \colon \mZ \times  \fG_{\vec{m}} \to \fG_{\vec{m}}$ and  $\Phi_{\vec{n}} \colon \mZ \times  \fG_{\vec{n}} \to \fG_{\vec{n}}$ are conjugate by an automorphism $\theta \colon \fG_{\vec{m}} \to \fG_{\vec{n}}$. Then by Proposition~\ref{prop-inducedhomeos}, the suspension spaces 
 $\fM(\Sigma_{g_1} , a_1 , \vec{m})$ and $\fM_2 = \fM(\Sigma_{g_1} , a_1 , \vec{n})$ are homeomorphic. 
  
  To show the converse in 2) assume that a homeomorphism $H$ exists, and suppose   that for some prime $p$ we have $C_{\vec{m}}(p) \ne C_{\vec{n}}(p)$.
  We assume without loss of generality that $C_{\vec{m}}(p) < C_{\vec{n}}(p)$. If otherwise, then reverse the roles of $\vec{m}$ and $\vec{n}$ and consider the homeomorphism $H^{-1}$.  
  Then as $\chi(\Sigma_1) = \chi(\Sigma_2)$, for sufficiently large $k$ the prime factorization of the Euler characteristic $\chi(M_{\ell_k})$ contains a lower power of $p$ than the prime factorization of $\chi(N_k)$.  
 But this  contradicts the fact that $\chi(M_{\ell_k}) =d_k \cdot \chi(N_k)$ where $d_k$ is the covering degree of $\wtH_k$. 

 Finally, to show 3) let $\Sigma = \Sigma_{g_1} = \Sigma_{g_2}$ where $g = g_1 = g_2 > 1$.
It suffices to chose    $\vec{m}$, $\vec{n}$ such that $\vec{m} \, \mor \, \vec{n}$, but   $C_{\vec{m}} \ne C_{\vec{n}}$. It then follows from 2) that $\fM_1 \not\approx \fM_2$. Pick a prime $p_1 \geq 3$ and let $\vec{m}$ be any sequence such that $C_{\vec{m}}(p_1) = 0$. Then define $\vec{n}$ by setting
$n_1 = p_1$ and  $n_{k+1} = m_k$ for all $k \geq 1$. 
 
 Note that $C_{\vec{m}}(p_1) = 0 \ne 1 = C_{\vec{n}}(p_1)$,  so $C_{\vec{m}}(p) \ne C_{\vec{n}}(p)$ is satisfied. But clearly $\vec{m} \, \mor \, \vec{n}$, so the \emph{adic}-surfaces $\fM(\Sigma_{g} , a_1, \vec{m})$ and $\fM(\Sigma_{g} , a_1, \vec{n})$ are return equivalent by Proposition~\ref{prop-surfadicMor}, but are not homeomorphic by  part 2) above. 
\endproof

\subsection{$\vec{m}$-adic solenoids of   higher dimension} \label{subsec-higher}

 Observe that the   requirements on the base manifold $\Sigma$ used in the proofs of 2) and 3) of  Theorem~\ref{thm-surfaces} are that:
 \begin{enumerate}
\item $\Sigma$ is a strongly Borel manifold, so that the maps $\wtH_{k}$ can be assumed to be coverings;
\item the fundamental group $G_0 = \pi_1(\Sigma, x)$ admits an epimorphism onto $\mZ$, or equivalently that $H^1(\Sigma ; \mZ)$ contains a copy of $\mZ$;
\item the Euler characteristic $\chi(\Sigma) \ne 0$.
\end{enumerate}
Thus, the proof of parts 2) and 3) of Theorem~\ref{thm-surfaces}  can be applied almost verbatim to show:
\begin{thm}\label{thm-main4}
Let $M$ be a closed manifold of dimension $n \geq 3$. Assume that $M$ is strongly Borel, that $H^1(M ; \mZ)$ has rank at least $1$, and the Euler characteristic $\chi(M) \ne 0$. 
Let $\fM_1 = \fM(M , a , \vec{m})$ and $\fM_2 = \fM(M , a , \vec{n})$ be the corresponding \emph{adic}-solenoids, where $a \colon \pi_1(M,x) \to \mZ$ is an epimorphism. Then we have:
\begin{enumerate}
\item   $\fM_1$ and $\fM_2$ are   homeomorphic if and only if   $C_{\vec{m}}=C_{\vec{n}}$.
\item there exists $\vec{m} \mor \vec{n}$   but  $\fM_1 \not\approx \fM_2$. 
\end{enumerate}
\end{thm}
 
 Finally, we comment on the requirement in   Theorem~\ref{thm-main3} that the base manifolds be strongly Borel. Let $M$ be a closed $n$-manifold where $n \geq 5$. Suppose that $M$ satisfies the conditions of Theorem~\ref{thm-main4}. 
 
 Let $N = M \# \mS^2 \times \mS^{n-2}$ be the closed $n$-manifold obtained by attaching the handle $\mS^2 \times \mS^{n-2}$. Then $\pi_1(M, x) \cong \pi_1(N,x)$ where we choose the basepoint   $x \in M$ disjoint from the disk along which the handle is attached. 
 
 Form the \emph{adic}-solenoids $\fM_1 = \fM(M , a , \vec{m})$ and $\fM_2 = \fM(N , a , \vec{m})$ as before, but with bases $M$ and $N$. Then $\fM_1$ and $\fM_2$ are return equivalent, as in fact they have conjugate global monodromy actions. On the other hand, all leaves in $\fM_1$ have trivial higher homotopy groups, while all leaves in $\fM_2$ have non-trivial higher homotopy groups. Thus, $\fM_1$ and $\fM_2$ can not be homeomorphic.



\begin{thebibliography}{10}

\bibitem{AartsFokkink1991}
{J.M.~Aarts and R.J.~Fokkink},
\newblock {\it The classification of solenoids},
\newblock {\bf Proc. Amer. Math. Soc.}, 111 :1161-1163, 1991.

\bibitem{AO1995}
{J.~Aarts and L.~Oversteegen},
\newblock {\it Matchbox manifolds},
\newblock In {\bf Continua ({C}incinnati, {OH}, 1994)},
\newblock {Lecture Notes in Pure and Appl. Math., Vol. 170},
\newblock {Dekker, New York}, 1995, pages 3--14..
 
\bibitem{AP1998}
{J.~Anderson and I.~Putnam},
\newblock {\it Topological invariants for substitution tilings and their associated {$C\sp *$}-algebras},
\newblock {\bf Ergodic Theory Dyn. Syst.}, 18:509--537, 1998.

\bibitem{BS2007}
{M.~Barge and R.~Swanson},
\newblock {\it New techniques for classifying {W}illiams solenoids},
\newblock {\bf Tokyo J. Math.}, 30:139--157, 2007.
\newblock{\bf Erratum} Tokyo J. Math. 34:287--288, 2011. 

\bibitem{BargeSadun2011}
{M.~Barge and R.~Swanson},
\newblock {\it Quotient cohomology for tiling spaces},
\newblock {\bf New York J. Math.}, 17:579--599, 2011.

\bibitem{Bing1960}
{R.H.~Bing},
\newblock {\it A simple closed curve is the only homogeneous bounded plane continuum that contains an arc},
\newblock {\bf Canad. J. Math.}, 12:209--230, 1960.

\bibitem{BlockKeesling2004}
{L.~Block and J.~Keesling},
\newblock {\it A characterization of adding machine maps},
\newblock {\bf Topology Appl.}, 140:151--161, 2004.

\bibitem{CN1985}
{C.~Camacho and A.~Lins~Neto},
\newblock {\bf Geometric Theory of Foliations},
\newblock {Translated from the Portuguese by Sue E. Goodman},
\newblock {Progress in Mathematics}, {Birkh\"auser Boston, MA}, 1985.

\bibitem{CandelConlon2000}
{A.~Candel and L.~Conlon},
\newblock {\bf Foliations I},
\newblock Amer. Math. Soc., Providence, RI, 2000.

\bibitem{ClarkHurder2011a}
{A.~Clark and S.~Hurder},
\newblock {\it Embedding solenoids in foliations},
\newblock {\bf Topology Appl.}, 158:1249--1270, 2011.

\bibitem{ClarkHurder2013}
{A.~Clark and S.~Hurder},
\newblock {\it Homogeneous matchbox manifolds},
\newblock {\bf Trans. Amer. Math. Soc.}, 365:3151--3191, 2013.

\bibitem{CHL2013c}
{A.~Clark, S.~Hurder and O.~Lukina},
\newblock {\it Classifying matchbox manifolds},
\newblock {preprint}, August 2013, {arXiv:1311.0226}.

\bibitem{CHL2014}
{A.~Clark, S.~Hurder and O.~Lukina},
\newblock {\it Shape of matchbox manifolds},
\newblock {\bf Indag. Math.}, 25(4):669--712, 2014.

\bibitem{CHL2017a}
{A.~Clark, S.~Hurder and O.~Lukina},
\newblock {\it Manifold-like matchbox manifolds},
\newblock {preprint}, 2017; {arXiv:1704.04402}.

\bibitem{ClarkHunton2016}
{A.~Clark  and J.~Hunton},
\newblock {\it The homology core and invariant measures},
\newblock {\it Transactions A.M.S.}, 2017, {\text DOI: https://doi.org/10.1090/tran/7398}.
 
 
\bibitem{CortezPetite2008}
{M.-I.~Cortez and S.~Petite},
\newblock {\it $G$-odometers and their almost one-to-one extensions},
\newblock {\bf J. London Math. Soc.}, 78(2):1--20, 2008. 

\bibitem{CortezMedynets2016}
{M.I.~Cortez and K.~Medynets},
\newblock {\it Orbit equivalence rigidity of equicontinuous systems}
\newblock {\bf  J. Lond. Math. Soc. (2)}, 94:545--556, 2016.
  
\bibitem{Davis2012}
{J.F.~Davis},
\newblock {\it The work of Tom Farrell and Lowell Jones in topology and geometry},
\newblock {\bf  Pure Appl. Math. Q.}, 8:1--14, 2012.

\bibitem{DHL2016a}
{J.~Dyer, S.~Hurder and O.~Lukina},
\newblock {\it The discriminant invariant of Cantor group actions},
\newblock {\bf Topology Appl.}, 208: 64-92, 2016.


\bibitem{DHL2016b}
{J.~Dyer, S.~Hurder and O.~Lukina},
\newblock {\it Growth and homogeneity of matchbox manifolds},
\newblock {\bf Indagationes Math.}, 28:145--169, 2017.

\bibitem{DHL2016c}
{J.~Dyer, S.~Hurder and O.~Lukina},
\newblock {\it Molino theory for matchbox manifolds},
\newblock {\bf Pacific J. Math.}, 289:91-151, 2017.

 \bibitem{EMT1977}
{D.B.A.}~Epstein, {K.C.}~Millet, and {D.}~Tischler,
\newblock {\it Leaves without holonomy},
\newblock {\bf Jour. London Math. Soc.}, 16:548--552, 1977.
 
\bibitem{FO2002}
{R.~Fokkink and L.~Oversteegen},
\newblock {\it Homogeneous weak solenoids},
\newblock {\bf Trans. Amer. Math. Soc.}, 354(9):3743--3755, 2002.

\bibitem{FHK2002}
{A.~Forrest, J.~Hunton and J.~Kellendonk},
\newblock {\bf Topological invariants for projection method patterns},
\newblock {Mem. Amer. Math. Soc.}, Vol. 159, 2002.

\bibitem{Ghys1999}
{\'{E}~Ghys},
\newblock {\it Laminations par surfaces de Riemann},
\newblock in {\bf Dynamique et G\'{e}om\'{e}trie Complexes},
\newblock {Panoramas \& Synth\`{e}ses}, 8:49--95, 1999.

\bibitem{GPS2017}
{T.~Giordano, I.~Putnam and C.~Skau},
\newblock {\it {$\mathbb Z^d$}-odometers and cohomology},
\newblock {preprint}, {arXiv:1709.08585v1}.  
    
\bibitem{Haefliger1984}
{A.~Haefliger},
\newblock {\it Groupo{\"\i}des d'holonomie et classifiants},
\newblock In {\bf Transversal structure of foliations (Toulouse, 1982)},
\newblock {Asterisque, 177-178, Soci\'et\'e Math\'ematique de France}, 1984:70--97.

\bibitem{Haefliger2002a}
{A.~Haefliger},
\newblock {\it Foliations and compactly generated pseudogroups}, 
\newblock in {\bf Foliations: geometry and dynamics (Warsaw, 2000)},
\newblock {World Sci. Publ., River Edge, NJ},   2002:275--295.

\bibitem{Hurder2008}
{S.~Hurder},
\newblock {\it Classifying foliations},
\newblock {\bf Foliations, Geometry and Topology. Paul Schweitzer Festschrift},
\newblock (eds. Nicoalu Saldanha et al),
\newblock Contemp. Math. Vol. 498, American Math. Soc., Providence, RI, 2009,  pages 1--61.

\bibitem{Hurder2014}
{S.~Hurder},
\newblock {\it Lectures on Foliation Dynamics},
\newblock {\bf Foliations: Dynamics, Geometry and Topology}, pages 87--149,
\newblock {Advanced Courses in Mathematics CRM Barcelona}, Springer Basel,   2014.

\bibitem{HL2017a}
{S.~Hurder and O.~Lukina},
\newblock {\it Wild solenoids},
\newblock {\it Transactions A.M.S.}, 2017, {\text DOI: https://doi.org/10.1090/tran/7339}.

  
\bibitem{JS2015}
{A.~Julien and L.~Sadun},
\newblock {\it Tiling deformations, cohomology, and orbit equivalence of tiling spaces}, {arXiv:1506.02694}. 
      
\bibitem{Luck2012}
{W. L\"{u}ck},
\newblock {\it Aspherical manifolds},
\newblock {\bf Bulletin of the Manifold Atlas}, 1--17, 2012.
\verb"http://131.220.77.52/lueck/publications.php#survey"
  
\bibitem{LM1997}
{M.~Lyubich and Y.~Minsky},
\newblock {\it Laminations in holomorphic dynamics},
\newblock {\bf J. Differential Geom.}, 47:17--94, 1997.

\bibitem{McCord1965}
{M.C.~McCord},
\newblock {\it Inverse limit sequences with covering maps},
\newblock {\bf Trans. Amer. Math. Soc.}, 114:197--209, 1965

\bibitem{MS2006}
{C.C.~Moore and C.~Schochet},
\newblock {\bf Analysis on Foliated Spaces},
\newblock Math. Sci. Res. Inst. Publ. vol. 9, Second Edition,
\newblock Cambridge University Press, New York, 2006.

\bibitem{Plante1975}
{J.~Plante},
\newblock {\it Foliations with measure-preserving holonomy},
\newblock {\bf Ann. of Math.}, 102:327--361, 1975.


\bibitem{RT1971}
{J. T. Rogers, Jr. and J. L. Tollefson},
\newblock{\it Maps between weak solenoidal spaces},
\newblock{\bf Coll. Math.} 23(2):245--249, 1971.

\bibitem{RT1971b} 
{J.T.~Rogers, Jr. and J.L.~Tollefson}, 
\newblock {\it Homogeneous inverse limit spaces with non-regular covering maps as bonding maps},
\newblock {\bf Proc. A.M.S} {\bf 29}: 417--420, 1971.

\bibitem{RT1972} 
{J.T.~Rogers, Jr. and J.L.~Tollefson}, 
\newblock {\it Involutions on solenoidal spaces}, 
\newblock {\bf Fund. Math.} {\bf 73} (1971/72), no. 1, 11--19.

\bibitem{Sadun2003}
{L.~Sadun}
\newblock {\it Tiling spaces are inverse limits},
\newblock {\bf J. Math. Phys.}, 44:5410--5414, 2003.

\bibitem{Sadun2008}
{L.~Sadun}
\newblock {\bf Topology of tiling spaces},
\newblock {University Lecture Series}, Vol. 46, American Math. Society, 2008.

\bibitem{Schori1966}
{R.~Schori},
\newblock {\it Inverse limits and homogeneity},
\newblock {\bf Trans. Amer. Math. Soc.}, 124:533--539, 1966.
    
\bibitem{Sullivan1988}
{D.~Sullivan},
\newblock{Bounds, quadratic differentials, and renormalization conjectures},
\newblock{American {M}athematical {S}ociety centennial publications, {V}ol. {II} ({P}rovidence, {RI}, 1988)},   {pages 417--466}, {1992}. 

\bibitem{Sullivan2014}
{D.~Sullivan},
\newblock{Solenoidal manifolds},
 \newblock {\bf J. Singul.}, 9:203--205, 2014.

  
\bibitem{vanDantzig1930}
{D.~van Dantzig},
\newblock {\it  \"{U}ber topologisch homogene {K}ontinua},
\newblock {\bf Fund. Math.}, 15:102--125, 1930.

\bibitem{Verjovsky2014}
{A.~Verjovsky},
\newblock{Commentaries on the paper {\it {S}olenoidal manifolds} by  {D}ennis {S}ullivan},
 \newblock {\bf J. Singul.}, 9:245--251, 2014.

\bibitem{Vietoris1927}
{L.~ Vietoris},
\newblock {\it  \"{U}ber den h\"oheren {Z}usammenhang kompakter {R}\"aume und eine {K}lasse von zusammenhangstreuen {A}bbildungen},
\newblock {\bf Math. Ann.}, 97:454--472, 1927.



\end{thebibliography}
\end{document}